\documentclass[11pt, a4paper]{amsart}
\usepackage{a4wide}
\usepackage{amssymb}
\usepackage{amsthm}
\usepackage{amsfonts}
\newtheorem{lemma}{\bf Lemma}[section]
\newtheorem{theorem}[lemma]{\bf Theorem}

\newtheorem{proposition}[lemma]{\bf Proposition}

\newtheorem{definition}[lemma]{\bf Definition}

\newcommand{\Pz}{{\mathcal P}_2}
\newcommand{\card}{{\mathrm{card}}}
\newcommand{\leqtree}{\leq_{T}}
\newcommand{\leqwo}{\leq_{\mathrm{wo}}}

\begin{document}

\title{A weak form of Hadwiger's conjecture}

\author{Dominic van der Zypen}
\address{M\&S Software Engineering, Morgenstrasse 129, CH-3018 Bern,
Switzerland}
\email{dominic.zypen@gmail.com}


\begin{abstract}
We introduce the following weak version of Hadwiger's conjecture:
If $G$  is a graph and $\kappa$ is a cardinal such that there is no coloring map
$c:G \to \kappa$, then $K_\kappa$ is a minor of $G$. We prove that this 
statement is true for graphs with infinite chromatic number.
\end{abstract}

\maketitle
\parindent = 0mm
\parskip = 2 mm
\section{Definitions}
In this note we are only concerned with simple
undirected graphs $G = (V, E)$ 
where $V$ is a set and $E \subseteq \Pz(V)$ where
$$\Pz(V) = \big\{ \{x,y\} : x,y \in V\textrm{ and } x\neq y\big\}.$$ 
We denote the vertex set of a graph $G$ by $V(G)$ and the edge set by 
$E(G)$. Moreover, for any cardinal $\alpha$ we denote the complete
graph on $\alpha$ points by $K_\alpha$. 

For any graph $G$, disjoint subsets $S, T \subseteq V(G)$ 
are said to be {\em connected
to each other} if there are $s \in S, t\in T$ with $\{s,t\}\in E(G)$.
Note that $K_\alpha$ is a {\em minor} of a graph $G$ if and only if 
there is 
a collection $\{S_\beta: \beta \in \alpha\}$ of nonempty, connected and 
pairwise disjoint subsets of $V(G)$ such that for all $\beta,\gamma \in \alpha$
with $\beta \neq \gamma$ the sets $S_\beta$ and $S_\gamma$ are connected
to each other. 

Well-founded trees and well-founded tree decompositions as defined in \cite{RoSeTh} 
will be central later on:
\begin{definition}
A {\em well-founded tree} is a non-empty partially ordered set $T=(V,\leq)$
such that for every two elements $t_1, t_2$ their infimum exists and the
set $\{t'\in V: t' < t\}$ is a well-ordered chain for every $t\in V$.
For $t_1, t_2 \in V = V(T)$ we set $$T[t_1,t_2] = \{ t\in V(T): t\geq \inf\{t_1,t_2\}
\textrm{ and } (t \leq t_1 \textrm{ or } t \leq t_2) \}.$$
\end{definition}
\begin{definition}
A {\em well-founded tree-decomposition} of a graph $G$ is a pair $(T, W)$
where $T$ is a well-founded tree and $W: V(T) \to \mathcal{P}(V(G))$ is a map 
such that
\begin{itemize}
\item[\textrm(W1)] $V(G) = \bigcup \mathrm{im}(W)$, and $E(G) \subseteq
\bigcup \{\Pz(W(t)) : t \in V(T)\}$;
\item[\textrm(W2)] if $t'\in T[t_1, t_2]$ then $W(t_1)\cap W(t_2) \subseteq W(t')$;
\item[\textrm(W3)] if $C\subseteq V(T)$ is a chain with $c=\sup C \in V(T)$, 
then $\bigcap\{W(t): t\in C\} \subseteq W(c)$.
\end{itemize}
\end{definition}
Note that (W1) says that every vertex of $G$ is contained in some $W(t)$, and
every edge has both its endpoints in some $W(t)$. 
\begin{definition}
We say that a well-founded tree-decomposition has {\em width} $<\kappa$ if for
every chain $C \subseteq V(T)$ we have
$$\card(\bigcup_{t\in C} \bigcap\{W(t'): t'\in C, t'\geq t\}) < \kappa.$$
\end{definition}
For the singleton chain $C=\{t\}$ this implies $\card(W(t)) < \kappa$ for
every $t\in V(T)$.

\section{The weak Hadwiger conjecture}
In \cite{Ha}, Hadwiger formulated his well-known and deep conjecture, 
linking the chromatic number $\chi(G)$ of a graph $G$ with 
clique minors. He conjectured that if $\chi(G) = n \in \mathbb{N}$ 
then $K_n$ is a minor of $G$. However for graphs with infinite 
chromatic number, the conjecture does not hold: in \cite{me} a graph
$G$ is given such that $\chi(G) = \omega$, but $K_\omega$ is not
a minor of $G$.

We consider the following weaker form of Hadwiger's conjecture:
\begin{quote}\textbf{Weak Hadwiger Conjecture.} {\sl Let $G$ be a graph
and $\kappa$ be a cardinal such that there is no coloring map
$c: G\to \kappa$. Then $K_\kappa$ is a minor of $G$.} 
\end{quote}

Note that in the finite case, this statement translates to: if $\chi(G) = n$
then $K_{n-1}$ is a minor of $G$. As of now, it seems to be an open
problem whether the weak Hadwiger conjecture is true in the finite case.

However in the infinite case, we can use
the following structure theorem by Robertson, Seymour, and Thomas:
\begin{theorem} \label{structure} \cite{RoSeTh} Let $\kappa$ be an infinite cardinal
and let $G$ be a graph. Then the following two conditions are equivalent:
\begin{enumerate}
\item $G$ contains no subgraph isomorphic to a subdivision of $K_\kappa$;
\item $G$ admits a well-founded tree-decomposition of width $< \kappa$.
\end{enumerate}
\end{theorem} 

The strategy is the following. We fix any graph $G$ and cardinal
$\kappa$ and assume that $K_\kappa$ is not a minor of $G$.
Then we construct a $\kappa$-coloring of $G$.

If $K_\kappa$ is not 
a minor of $G$, it is not a topological minor of $G$, which
is equivalent to condition (1) of \ref{structure}. So we apply
theorem \ref{structure} and it remains to prove the following statement:

\begin{proposition} Let $G$ be a graph with a well-founded tree-decomposition
of width $<\kappa$. Then there is a coloring map $c:G\to\kappa$.
\end{proposition}
\begin{proof}
It is sufficient to construct a mapping $f: V(G) 
\to \kappa$ such that the restriction $f|_{W(t)}$ is injective for every
$t\in T$: since
every edge lies entirely in some $W(t)$, the function $f$ will be a coloring of $G$.

We set $X := V(G)$. Denote the ordering relation on $T$ by $\leqtree$. 
It is easy to see that $\leqtree$ can be extended to a total
well-ordering $\leqwo$ on $T$. Moreover, for $x \in X$ we define
$$m(x) = \min\{t\in T: x\in W(t)\},$$ where the minimum is
taken with respect to the well-ordering $\leqwo$ on $T$. (Note that the
minimum is taken over a non-empty set since $X = \bigcup_{t\in T}W(t)$.)
For $t\in T$ let $\varphi_t: W(t) \to \card(W(t)) < \kappa$ be a bijection.

Endow $X$ with a total well-ordering relation $\leq_X$ defined by
$$x \leq_X y \Leftrightarrow m(x) <_T m(y) \textrm{ or }
[m(x) = m(y) \textrm{ and } \varphi_{m(x)}(x) \leq \varphi_{m(y)}(y)].$$
We define $f: X\to \kappa$ recursively by
$$f(x) = \min \big( \kappa \setminus \{f(z): z <_X x 
\textrm{ and } z\in W(m(x))\}\big).$$
Note that the minimum above exists since $\kappa > \card(W(t))$ for
all $t\in T$.

It remains to show that for $t_0 \in T$ and $a\neq b \in W(t_0)$ we
have $f(a) \neq f(b)$.
Take any $a<_X b \in W(t_0)$. We consider the tree elements 
$m(a), m(b) \in T$. If $m(a) = m(b)$ then by the very definition of $f$
we get $f(a) \neq f(b)$ directly. 

So suppose that $m(a)\neq m(b)$. If $m(b) \not \leqtree t_0$ then consider
$i=\inf\{m(b), t_0\}$ in the tree. Clearly $i < m(b)$ and
because of axiom (W2) we have $b\in W(m(b))\cap W(t_0) \subseteq 
W(i)$, which contradicts the minimality of $m(b)$. 
Since the same argument can be made for $m(a)$ we get 
$$m(a), m(b) \leqtree t_0.$$

The definition of $\leq_X$ and the fact that
$a <_X b$ and $m(a)\neq m(b)$ jointly imply $m(a) \leqwo m(b)$.
Since predecessors of $t_0$ are linearly ordered in $\leqtree$ we have
$m(a) \leqtree m(b)$ or $m(b) \leqtree m(a)$.
Recall that $\leqwo$ extends $\leqtree$, so we get $m(a) <_T m(b)$.
Therefore $m(b) \in T[m(a), t_0]$ and we can apply axiom (W2) again
to get $$a \in W(m(a))\cap W(t_0) \subseteq W(m(b)).$$ Again we
go back to the recursive definition of $f$: we have 
$f(b)= \min \big( \kappa \setminus \{f(z): z <_X b 
\textrm{ and } z\in W(m(b))\}\big)$, and we get $f(b) \neq f(a)$
from the fact that $a\in W(m(b))$.
\end{proof}

\section{Acknowledgements}
I want to thank Robin Thomas for pointing out to me 
that the weak Hadwiger conjecture is implied by the 
structure theorem of \cite{RoSeTh} for graphs with
infinite chromatic number.

{\footnotesize
\begin{thebibliography}{99}
\bibitem{Ha} Hugo Hadwiger, {\it \"Uber eine Klassifikation der 
Streckenkomplexe}, Vierteljschr.~Naturforsch.~Ges.~Z\"urich, 
{\bf 88} (1943), 133--143.
\bibitem {RoSeTh} Neil Robertson, Paul D.~Seymour, Robin Thomas, {\it
Excluding subdivisions of infinite cliques}, Trans.~Amer.~Math.~Soc.~
({\bf 332}) (1992), no.~1, 211--223.
\bibitem{me} Dominic van der Zypen, {\it Hadwiger's conjecture for graphs with 
infinite chromatic number}, Advancement and Development in Mathematical
Sciences {\bf 4}, 2013, issue 1\&2, 1--4.
\end{thebibliography}
}
\end{document}